\documentclass[11pt]{amsart}
\newtheorem{theorem}{Theorem}
\newtheorem{corollary}{Corollary}
\newtheorem{lemma}{Lemma}
\theoremstyle{remark}
\newtheorem{remark}{Remark}
\newtheorem{problem}{Problem}
\newtheorem{example}{Example}

\begin{document}
\baselineskip17pt

\title[Competition]{Does a population with the highest turnover coefficient  win competition?}
\author[R. Rudnicki]{Ryszard Rudnicki}
\address{R. Rudnicki, Institute of Mathematics,
Polish Academy of Sciences, Bankowa 14, 40-007 Katowice, Poland.}
\email{rudnicki@us.edu.pl}


\begin{abstract}
We consider a discrete time competition model.
Populations  compete for common limited resources 
but they have different fertilities and mortalities rates.
We compare dynamical properties of this model  with its continuous counterpart.
We give sufficient conditions for competitive exclusion and the existence of periodic solutions 
related to the classical logistic, Beverton-Holt  and Ricker   models.
\end{abstract}

\maketitle

\section{Introduction}
\label{s:int}

It is well known that if species  compete for common limited resources, then usually they cannot coexist in the  long term.
This law was introduced by Gause \cite{Gause} and it is called \textit{the principle of competitive exclusion}.
There are a lot of papers where the problem of  competitive exclusion or  coexistence is discussed. Most of them are described by continuous time models, but 
there are also some number of discrete time models devoted to this subject  (see \cite{AS,CLCH} and the references given there).    
If we have continuous and discrete time versions of a similar competition model
 it is interesting to compare the properties of 
both versions of the model, especially to check if   
they are dynamically consistent,  i.e.,   if they  possess the same dynamical properties as stability or chaos.
In this paper we consider  a discrete time competition model with overlapping generations. We prove a sufficient condition  for competitive exclusion 
and compare it with its continuous counterpart. 
   
The model considered here is the following. A population consists of  $k$ different individual strategies. 
The state of the population at time $t$ is given  by the
vector $\mathbf x(t)=[x_{1}(t),\dots ,x_{k}(t)]$, where $x_{i}(t)$ is the
size of subpopulation with the phenotype $i$. 
Individuals with different phenotypes do not mate and 
each phenotype $i$ is characterized by per capita reproduction $b_i$ and mortality $d_i$.
We assume that  the  juvenile survival rate depends on the state $\mathbf x$ 
and it  is given by a function $f\colon \mathbb R^k_+ \to [0,1]$.
Therefore,  $f$ describes the suppression of growth driven, for example, by competition for food or 
free nest sites for newborns.
The time evolution of the state of  population 
is given by the system   
\begin{equation}
\label{d-m}
x_i(t+1)=x_i(t)-d_ix_i(t)+ b_i x_i(t)f(\mathbf x(t)).
\end{equation}  
We assume that $0<d_i\le 1$, $d_i<b_i$ for $i=1,\dots,k$ and $f(\mathbf x)>0$ for $\mathbf x\ne 0$.
The model is similar in spirit to that presented in \cite{AB}  (a continuous version) and in \cite{AR} (a discrete version) but
in those papers  $f$ has a special form strictly connected with competition for  free nest sites for newborns.
A simplified Leslie/Gower model \cite{AlSR} is also of the form~ (\ref{d-m}).
The suppression function $f$ can be quite arbitrary. 
Usually, it is of the form $f(\mathbf x)=\varphi (w_1x_1+\dots+w_kx_k)$, 
where $\varphi$ is a decreasing function and $w_1,\dots,w_k$ are positive numbers,
but e.g. in \cite{AB} it is of the form  
\[
f(\mathbf x)=
\begin{cases}
1,&\textrm{if  $\,(1+b_1)x_1+\dots+ (1+b_k)x_k\le K$},\\
 \dfrac{K-x_1-\dots-x_k}{b_1x_1+\dots+b_kx_k},
&\textrm{if  $\,(1+b_1)x_1+\dots +(1+b_k)x_k> K$}.
\end{cases}    
\]

Now we present some motivation for studying model (\ref{d-m}). 
We begin with a continuous time version of  it.
The time evolution of the state of population 
is described by the system   
\begin{equation}
\label{c-m}
x_i'(t)=-d_ix_i(t)+ b_i x_i(t)f(\mathbf x(t)),
\quad i=1,\dots,k. 
\end{equation}
We assume that  $0<d_i< b_i$,  $f$ has values in the interval $[0,1]$,
and 
\begin{equation}
\label{b:f}
f(\mathbf x) \le \min\bigg\{\frac{d_i}{b_i}
\colon \,\, i=1,\dots,k\bigg\}
\quad\textrm{if $|\mathbf x|\ge M$},
\end{equation}
where $|\mathbf x|=x_1+\dots+x_k$.
From the last condition it follows  that the total size $|\mathbf x(t)|$ 
of the population is bounded by $\max(M,|\mathbf x(0)|)$.
We also assume that $f$ is  
a sufficiently smooth
function to have the  existence and uniqueness of the solutions of (\ref{c-m}), for example it is enough to assume that $f$ satisfies a local Lipschitz condition.  
    
We denote by $L_i=b_i/d_i$ the turnover coefficient for the strategy $i$. 
We assume that
\begin{equation}
\label{ineq}
L_1>L_2\ge\dots \ge L_k.
\end{equation}
It is well known  that 
\begin{equation}
\label{goto0}
\lim_{t\to\infty} x_i(t)=0 \textrm{ \ for $i\ge 2$}.    
\end{equation}
Indeed, from (\ref{c-m}) it follows that 
\begin{equation}
\label{goto0-2}
(b_i^{-1}\ln x_i(t))'= -L_i^{-1}+f(\mathbf x(t)).    
\end{equation}
Thus 
\begin{equation}
\label{goto0-3}
(b_1^{-1}\ln x_1(t)-b_i^{-1}\ln x_i(t))'= L_i^{-1}-L_1^{-1}.    
\end{equation}
Therefore
\begin{equation}
\label{goto0-4}
\frac{d}{dt} \ln \bigg( \frac{x_1(t)^{b_i}}{x_i(t)^{b_1}}\bigg) = b_1b_2(L_i^{-1}-L_1^{-1})>0    
\end{equation}
and, consequently, 
\begin{equation}
\label{goto0-5}
\lim_{t\to\infty}\frac{x_1(t)^{b_i}}{x_i(t)^{b_1}}=\infty.
\end{equation}
Since $x_1(t)$ is a bounded function, from (\ref{goto0-5}) it follows that (\ref{goto0}) holds.

Now we return to a discrete time version of  model (\ref{c-m}). 
From (\ref{b:f}) it follows immediately that $x_i(t+1)\le x_i(t)$ if $|\mathbf x(t)|\ge M$ and, therefore,
the sequence $(x_i(t))$ is bounded.   Moreover, since $d_i \le 1$ we have $x_i(t)>0$ if $x_i(0)>0$.  
It is of interest to know whether   (\ref{goto0}) holds also
for discrete model  (\ref{d-m}).
Observe that (\ref{d-m})  can be written as 
\begin{equation}
\label{d-m1}
\frac{x_i(t+1)-x_i(t)}{b_ix_i(t)}= -L_i^{-1}+ f(\mathbf x(t)).
\end{equation} 
Now   (\ref{goto0-3})  takes the form
\begin{equation}
\label{d-m2}
 \frac{1}{b_1}D_l\, x_1(t)-\frac{1}{b_i}D_l\, x_i(t) = L_i^{-1}-L_1^{-1}.    
\end{equation}
In the last formula  the \textit{logarithmic derivative} $x_i'/x_i$ was replaced by its discrete version
\[
D_l\, x_i(t) :=\dfrac{x_i(t+1)-x_i(t)}{x_i(t)}.
\]
Let $\alpha=b_1/b_i$ and $\beta=b_1(L_i^{-1}-L_1^{-1})=\alpha d_i-d_1$. Then $0<\beta<\alpha$ and
 (\ref{d-m2}) can be written in the following way
\begin{equation}
\label{d-m3}
D_l\, x_1(t)= \alpha D_l\, x_i(t) +\beta.    
\end{equation}
We want to find a sufficient condition for (\ref{goto0}). 
In order to do it we formulate the following general question,
which can be investigated independently of the above biological models. 

\begin{problem}
\label{prob1}
Find parameters $\alpha$ and $\beta$, $0<\beta<\alpha$, such that the following condition 
holds:\\
(C) \ if $(u_n)$  and $(v_n)$ are arbitrary bounded sequences  of positive numbers  
 satisfying
\begin{equation}
\label{p1}
\frac{u_{n+1}-u_n}{u_n}= \alpha \frac{v_{n+1}-v_n}{v_n} +\beta,    
\end{equation}
for $n\in\mathbb{N}$, then $\lim\limits_{n\to\infty} v_n=0$.
\end{problem}

In the case when the model has the property of competitive exclusion (\ref{goto0}) one can ask if the dynamics of the $k$-dimensional model 
is the same as in the restriction to the one-dimensional model. 
The answer to this question is positive for the continuous version,
because the one-dimensional model has a very simple dynamics. In Section~\ref{ss:one} we also show that both dynamics are similar if 
the one-dimensional model has the shadowing property. More interesting question is what can happen when condition (C) does not hold.
One can expect  that then subpopulations with different strategies can coexist even if condition (\ref{ineq}) holds. 
But we do not have a coexistence equilibrium 
(i.e. a positive stationary solution of (\ref{d-m})) which makes the problem 
more difficult. In Section~\ref{ss:periodic} we check that two-dimensional systems  with  $f$ related to the classical logistic,
Beverton-Holt  and Ricker models can have periodic solutions even in the case
 when one-dimensional versions of these models have  
stationary globally stable solutions and the two-dimensional model has  a locally stable boundary equilibrium $(x_1^*,0)$. 

\section{Competitive exclusion}
\label{s:mr}
The solution of Problem~\ref{prob1} is formulated in the following theorem.
  
\begin{theorem}
\label{th1}
If \ $\alpha\le 1+\beta$  then condition (C) is fulfilled.
If $\alpha> 1+\beta$
then  we can find  periodic sequences $(u_n)$ and $(v_n)$ of period two with positive elements
which satisfy $(\ref{p1})$.
\end{theorem}

\begin{lemma}
\label{l:c}
Consider the function
\begin{equation}
\label{l-1}
g_n (x_1,\dots,x_n)=(\alpha x_1+\gamma)(\alpha x_2+\gamma)\cdots(\alpha x_n+\gamma)
\end{equation}
defined on the set 
$S_{n,m}=\{\mathbf x\in \mathbb R^n_+\colon \,x_1x_2\cdots x_n=m\}$,
where $\alpha>0$, $\gamma\ge 0$, $m>0$ and $n$ is a positive integer. 
Then 
\begin{equation}
\label{l-2}
g_n (x_1,\dots,x_n)\ge \left(\alpha m^{1/n}+\gamma \right)^n.
\end{equation}
\begin{proof}
The case $\gamma=0$ is obvious, so we assume that $\gamma>0$.
We use the standard technique of 
the Lagrange multipliers for investigating problems on conditional extrema.
Let
\[
L(x_1,\dots,x_n)=g_n (x_1,\dots,x_n)+\lambda(m-x_1x_2\cdots x_n).
\]
Then 
\[
\frac{\partial L(x_1,\dots,x_n)}{\partial x_i}
=\frac{\alpha g_n (x_1,\dots,x_n)}{\alpha x_i+\gamma}
-\frac{\lambda x_1x_2\cdots x_n}{x_i}.
\]
Observe that if 
$\dfrac{\partial L(x_1,\dots,x_n)}{\partial x_i}=0$ for $i=1,\dots,n$,
then $x_1=\dots=x_n=m^{1/n}$. It means that the function $g_n$
has only one  local conditional extremal point
and this point is the global  minimum because $g_n(\mathbf x)$ converges to infinity as  $\|\mathbf x\|\to \infty$.
\end{proof}

\end{lemma}
\begin{proof}[Proof  of Theorem~\ref{th1}] 
Equation (\ref{p1}) can be written in the following form  
\begin{equation}
\label{p2}
\frac{u_{n+1}}{u_n}= \alpha \frac{v_{n+1}}{v_n} +\gamma,    
\end{equation}
where $\gamma=\beta+1-\alpha$. 
Consider the case $\alpha\le 1+\beta$. Then $\gamma\ge 0$ and $\alpha+\gamma=\beta+1>1$.
We show that if  $(v_n)$ is a sequence of positive numbers such that $\limsup\limits_{n\to\infty} v_n>0$ and $(u_n)$
is a sequence of positive numbers such that (\ref{p2}) holds, 
then the sequence $(u_n)$ is unbounded.
Indeed, then we can find an $\overline{m}>0$ and a subsequence  
$(v_{n_i})$ of $(v_n)$ such that  $v_{n_i}\ge \overline m$ for $i\in\mathbb N$.
We set  $v_0=1$ and  $x_i=v_{i}/v_{i-1}$ for $i\in\mathbb N$. Then 
$v_{n}=x_1\cdots x_n$ and $u_n=u_0g_n(x_1,\dots,x_n)$, where 
$u_0=u_1/(\alpha v_1+\gamma)$. 
If  $m=x_1\cdots x_{n_i}$, then $m\ge \overline m$ and from
Lemma~\ref{l:c} it follows that
\[
u_{n_i}=u_0g_{n_i}(x_1,\dots,x_{n_i})\ge 
u_0\left(\alpha m^{1/n_i}+\gamma \right)^{n_i}\ge 
u_0\left(\alpha \overline m^{1/n_i}+\gamma \right)^{n_i}.
\]
Since $\lim\limits_{i\to\infty}\overline m^{1/n_i}=1$ and $\alpha+\gamma>1$
we obtain
$\lim\limits_{i\to\infty} u_{n_i}=\infty$, which proves the first part of the theorem. 
Now we assume that $\alpha> 1+\beta$. Then $\gamma<0$.
First we check that there exists $\theta>1$ such that
\begin{equation}
\label{per1} 
(\alpha \theta+\gamma)(\alpha \theta^{-1}+\gamma)=1.
\end{equation}
Equation (\ref{per1}) can be written in the following form
\begin{equation}
\label{per3} 
\theta+\theta^{-1}=L, \textrm{  \, where $L=\frac{\alpha^2+\gamma^2-1}{\alpha|\gamma|}$}.
\end{equation}
Since $\alpha+\gamma=\beta+1>1$ we have
$\alpha^2+\gamma^2-1> 2\alpha|\gamma|$, 
which  gives $L>2$ and implies that  there exists $\theta>1$ such that 
(\ref{per1}) holds.
Now we put
$u_{2n-1}=c_1$, $u_{2n}=c_1(\alpha \theta+\gamma)$,
$v_{2n-1}=c_2$,  $v_{2n}=c_2\theta$ 
for $n\in\mathbb N$, where $c_1$ and $c_2$ are any positive constants. Then
\[
\frac{u_{2n}}{u_{2n-1}}= \alpha \theta+\gamma=\alpha\frac{v_{2n}}{v_{2n-1}}+\gamma,
\]
and using (\ref{per1}) we obtain
\[
\frac{u_{2n+1}}{u_{2n}}= \frac1{\alpha \theta+\gamma}
=\alpha \theta^{-1}+\gamma=\alpha\frac{v_{2n+1}}{v_{2n}}+\gamma.
\qedhere
\]
\end{proof} 
\begin{remark}
\label{r:1}
We have proved 
a slightly  stronger condition than (C) in the case $\alpha\le 1+\beta$. Namely, 
if $(u_n)$ is a bounded sequence  of positive numbers, 
$(v_n)$ is a  sequence  of positive numbers and they satisfy
(\ref{p1}), then $\lim\limits_{n\to\infty} v_n=0$.
In the proof of condition (C) we have not used the preliminary assumption that $\beta<\alpha$.
\end{remark}

\section{Applications}
\label{s:appl}
Now we return to the model given  by (\ref{d-m}).
We assume that $f\colon \mathbb R^k_+\to [0,1]$ is a continuous function
which satisfies (\ref{b:f}). From (\ref{b:f}) it follows that there exists $\overline{M}>0$
such that the set    
\[
X=\{\mathbf x\in\mathbb R_+^k\colon  x_1+\dots+x_k\le \overline M\} 
\]  
is invariant under (\ref{d-m}), i.e., if 
$\mathbf x(0)\in X$ then $\mathbf x(t)\in X$ for $t>0$. We  restrict the domain of the model   
to the set $X$.  Let $T\colon X\to X$ be the transformation given by 
$T_i(\mathbf x)=(1-d_i)x_i+b_if(\mathbf x)x_i$, for $i=1,\dots,k$.
 
\subsection{Persistence}
\label{ss:persistence}
First we check that if $f(\mathbf  0)=1$ then the population is \textit{persistent}, i.e.,
 $\liminf_{n \to\infty}\|T^n(\mathbf x)\|\ge \varepsilon_1 >0$ for all $\mathbf x\ne \mathbf 0$. 
This is a standard result from persistence theory but we check it to make the paper self-contained.
 Since $b_i>d_i$ for  $i=1,\dots,k$ we find $\varepsilon>0$ and   $\delta>0$ 
such that 
\begin{equation}
\label{ej-f-p}
T_i(\mathbf x)\ge  (1+\delta)x_i \quad\textrm{for $i=1,\dots,k$ and $\mathbf x \in B(\mathbf 0,\varepsilon)$}, 
\end{equation}
where $B(\mathbf 0,\varepsilon)$ denotes the open ball in $X$ with center $\mathbf  0$ and radius $\varepsilon$.
Moreover, since $T(\mathbf x)\ne \mathbf 0$ for $\mathbf x\ne 0$ the closed set $T(X\setminus B(\mathbf 0,\varepsilon))$ is  disjoint with some 
neighbourhood of $\mathbf 0$. Using (\ref{ej-f-p}) we find $\varepsilon_1\in (0,\varepsilon)$ such that 
for each $\mathbf x \ne \mathbf 0$ we also find   an integer  $n_0(\mathbf x)$ such
that $T^n(\mathbf x)\notin  B(\mathbf 0,\varepsilon_1)$ for $n\ge n_0(\mathbf x)$.

\subsection{Convergence to one-dimensional dynamics}
\label{ss:one}
Now we present some corollaries of Theorem~\ref{th1} concerning the long-time  behaviour of the population.  
The inequality $0<\alpha\le 1+\beta$ can be written in terms of birth and death coefficients as
\begin{equation}
\label{a:1} 
b_1(1-d_i)\le b_i(1-d_1)\quad\textrm{for $i=2,\dots,k$.} 
\end{equation}
It means that if (\ref{ineq}) and (\ref{a:1}) hold, then
all strategies except the first one become extinct.
It suggests that the model should behave asymptotically as $t\to\infty$, like a one-dimensional model corresponding 
to a population consisting of only the first strategy. This reduced model is given by the 
recurrent formula 
\begin{equation}
\label{d-m-a}
y(t+1)=S(y(t)),
\end{equation} 
where $S(y)=y-d_1y+ b_1yf(y,0,\dots,0)$.

In order to check that the model  given  by (\ref{d-m}) has the same asymptotic behaviour 
as the transformation $S$, we need some auxiliary definitions.
A sequence $(y_k)$ is called an $\eta$-\textit{pseudo-orbit} of a transformation $S$ if
$|S(y_k)- y_{k+1}|<\eta$ 
for all $k\ge 1$. The transformation $S$ is called \textit{shadowing}, if for
every $\delta>0$ there exists
$\eta>0$ such that for each 
$\eta$-pseudo-orbit $(y_k)$ of $S$ 
there is a point $y$ such that 
$|y_k -S^k(y)|<\delta$.

\begin{theorem}
\label{th-as}
Assume that  $f(\mathbf 0)=1$ and that conditions  $(\ref{ineq})$, $(\ref{a:1})$ hold.
Then 
\[
\lim_{t\to\infty} x_i(t)=0 \ \textrm{ for $\,i=2,\dots,k$}.
\] 
If $S$ is shadowing then for each $\delta>0$ 
and for each initial point $\mathbf x(0)=(x_1,\dots,x_k)$
with $x_1>0$ there exists  $t_0\ge 0$  such that
$y(t_0)>0$ and 
\begin{equation}
\label{wn-sh}
\big|
x_1(t)
- y(t)
\big|
<\delta
\textrm{ \ for $t\ge t_0$}. 
\end{equation}
\end{theorem} 

\begin{proof}
Let us fix a $\delta>0$ and let $\eta>0$ be a constant from the shadowing property of $S$.
From the uniform continuity of the function $f$  
there is an $\varepsilon>0$ such that 
\begin{equation}
\label{a:t1}
\overline M b_1\big|f(x_1,\dots,x_k)-f(x_1,0,\dots,0)
\big|
<\eta
\textrm{ if $\,\mathbf x\in X$, $x_2+\dots+x_k<\varepsilon$.}
\end{equation}
Since
all strategies except the first one become extinct,
there exists $t_0\ge 0$ such that $x_2(t)+\dots+x_k(t)<\varepsilon$ for $t\ge t_0$.  
From (\ref{a:t1}) it follows that
\[
|x_1(t+1)-S(x_1(t))|<\eta  \textrm{ \ for $t\ge t_0$}
\]
and, consequently, 
 the sequence $x_1(t_0),x_1(t_0+1),\dots$ is 
 an $\eta$-pseudo-orbit.
Since $S$ is shadowing we have
(\ref{wn-sh}).  
\end{proof}

The shadowing property was intensively studied for the last thirty years 
and there are a lot of results concerning the shadowing property for
one-dimensional maps (cf. a survey paper by Ombach and Mazur \cite{OmbachMazur}).
It is obvious that if $S$ has an asymptotically stable periodic orbit then
$S$ is shadowing on the basin of attraction of this orbit.
Moreover, for a continuous one-dimensional  transformation 
the convergence of all iterates to a unique fixed point $x$ implies its global stability
 \cite{Sedeghat}. 
 Thus, as a simple consequence of Theorem~\ref{th-as} we obtain
\begin{corollary}
\label{cor-as}
Assume that  $f(\mathbf 0)=1$ and that conditions  $(\ref{ineq})$, $(\ref{a:1})$ hold.
If $S$ has a fixed point $x_*>0$ and $\lim\limits_{n\to\infty}S^n(x)=x_*$ for all $x>0$,
then for each initial point $\mathbf x(0)=(x_1,\dots,x_k)$
with $x_1>0$, we have 
$\lim\limits_{t\to\infty}\mathbf x(t)=(x_*,0,\dots,0)$.
\end{corollary} 
Some applications of shadowing to semelparous population 
similar to Theorem~\ref{th-as} and Corollary~\ref{cor-as}     
can be found in \cite{RudnickiWieczorek2010}.
An interested reader will find there also some observations concerning     
chaotic behaviour of such models. In particular, 
the model given  by (\ref{d-m}) can exhibit chaotic behaviour if 
the suppression function is of the form
$f(\mathbf x)=1-x_1-\dots-x_n$, i.e., it is a generalization of the logistic model. 

\begin{remark}[Dynamical consistence]
\label{r:d-c}
If we replace $x'_i(t)$ with   $(x_i(t+h)-x_i(t))/h$ in  (\ref{c-m}) then we get 
\begin{equation}
\label{d-mh}
x_i(t+h)=x_i(t)-d_ihx_i(t)+ b_ih x_i(t)f(\mathbf x(t)),\quad i=1,\dots,k.
\end{equation}
One can ask if this scheme is dynamically consistent with (\ref{c-m}).
Observe, that inequalities (\ref{ineq}) also hold if we replace  $b_i$ with $b_ih$ and $d_i$ with $d_ih$.
The difference equation  (\ref{d-mh}) is said to be \textit{dynamically consistent} with 
(\ref{c-m}) if they  possesses the same dynamical
behavior such as local stability, bifurcations, and chaos \cite{LE},
or more specifically \cite{Mickens,RG} if they have the same given property,
e.g. if the competitive exclusion takes place in both discrete and continuous models.
The model  (\ref{d-mh}) is biologically meaningful only if the death coefficients are $\le 1$, i.e., if
\begin{equation}
\label{w:h}    
0<h\le h_{\max{}}=\min \{d_1^{-1},\dots,d_k^{-1}\}.  
\end{equation}
We assume that $b_i$ and $d_i$ satisfy (\ref{ineq}), i.e., $b_1d_i>b_id_1$ for $i=2,\dots,k$.
Let $b_{i,h}=b_ih$, $d_{i,h}=d_ih$.
Now, (\ref{a:1}) applied to $b_{i,h}$ and $d_{i,h}$ gives
\begin{equation}
\label{a:1-r} 
b_1-b_i\le (b_1d_i-b_id_1)h\quad\textrm{for $i=2,\dots,k$}.
\end{equation}
In particular if (\ref{ineq}) holds and $b_i\ge b_1$ for $i=2,\dots,k$ then  for all $h$ satisfying (\ref{w:h})
all strategies except the first one become extinct, i.e., the difference equation (\ref{d-mh}) is dynamically consistent  with 
(\ref{c-m}) with respect to this property. We cannot expect ``full" dynamical consistence if the above conditions hold,  because in the case of the logistic map,
 i.e., if $f(\mathbf x)=1-x_1-x_2$, the stationary point $\mathbf  x_1^*=((b_1-d_1)d_1^{-1},0)$ of  (\ref{c-m}) is globally stable but in the numerical scheme 
 (\ref{d-mh}) this point  loses stability when $b_1h>2+d_1h$.
\end{remark}

\subsection{Periodic solutions}
\label{ss:periodic}
Theorem~\ref{th1} can be also useful if we look for periodic oscillation in 
the model  given  by (\ref{d-m}). 
We restrict our investigation to the two-dimensional model. 
We recall that
if  $\alpha> 1+\beta$, then 
the periodic sequences  given by 
$u_{2n-1}=c_1$,
$u_{2n}=c_1(\alpha \theta+\gamma)$,
$v_{2n-1}=c_2$,  $v_{2n}=c_2 \theta$  
 for $n\in\mathbb N$, satisfy $(\ref{p1})$.
Here $c_1$ and $c_2$ are any positive constants,
$\theta>1$ is a solution of the equation
\[
(\alpha \theta+\gamma)(\alpha \theta^{-1}+\gamma)=1,
\]
$\alpha=b_1/b_2$, $\beta=\alpha d_2-d_1>0$, and $\gamma=1+\beta-\alpha=\alpha(d_2-1)+(1-d_1)<0$.
Under these assumptions we are looking for $c_1, c_2>0$ such that
\begin{equation}
\left\{
\label{ukl-2'}
\begin{aligned}
&\theta=1-d_2+b_2f(c_1,c_2),
\\
&1=\theta(1-d_2)+b_2 \theta  f(c_1(\alpha \theta+\gamma),c_2\theta).
\end{aligned}
\right.
\end{equation}
This system is equivalent to
\begin{equation}
\label{ukl-3'}
\left\{
\begin{aligned}
&f(c_1,c_2)=\left(\theta+d_2-1\right)b_2^{-1}
\\
&f(c_1(\alpha \theta+\gamma),c_2\theta)
=\left(\theta^{-1}+d_2-1\right)b_2^{-1}.
\end{aligned}
\right.
\end{equation}
Since $f(\mathbf  x)\in (0,1)$ for $\mathbf x\in X\setminus\{\mathbf 0\}$,
we have the following necessary condition for the existence of positive solution 
of the system (\ref{ukl-3'}) : 
 \begin{equation}
\label{e:wko}
\theta<1+b_2-d_2\quad\textrm{and}\quad \theta<(1-d_2)^{-1}.
\end{equation}

Let $f(\mathbf x)=\varphi(x_1+x_2)$,  where
$\varphi$ is a strictly decreasing function defined on the interval $[0,K)$, $0<K\le \infty$,  
such that $\varphi(0)=1$ and $\lim_{x\to K}\varphi(x)=0$.   
Define
$m_1=\left( \theta+d_2-1\right)b_2^{-1}$, 
 $m_2=\left(\theta^{-1}+d_2-1\right)b_2^{-1}$, $p_1=\varphi^{-1}(m_1)$,
 and $p_2=\varphi^{-1}(m_2)$. If (\ref{e:wko}) holds then the constants $p_1$, $p_2$
 are well defined and $0<p_1<p_2$. 
Thus, we find a positive solution of  system (\ref{ukl-3'}) if and only if
(\ref{e:wko}) holds  and  
 \begin{equation}
\label{e:wkw}
c_1+c_2=p_1 \quad\textrm{and}\quad
c_1(\alpha \theta+\gamma)+c_2\theta=p_2.
\end{equation}
System (\ref{e:wkw}) has a unique solution
 \begin{equation}
\label{e:wkw2}
c_1=
 \frac{p_2-p_1\theta}
{\alpha \theta+\gamma-\theta},
\quad
c_2=
 \frac{p_1(\alpha \theta+\gamma)-p_2}
{\alpha \theta+\gamma-\theta}.
\end{equation}
Since $\alpha>1$, $\theta>1$, and $\beta>0$ we have
\[
\alpha \theta+\gamma-\theta=\alpha \theta+1+\beta-\alpha-\theta=(\alpha-1)(\theta-1)+\beta>0.
\]
Thus 
system (\ref{ukl-3'}) has a positive solution if and only if  
(\ref{e:wko}) holds and
\begin{equation}
\label{e:wkw3}
p_1\theta< p_2<p_1(\alpha \theta+\gamma).
\end{equation}

Now we show how to find parameters $b_1,b_2,d_1,d_2$ such that 
(\ref{e:wko}) and (\ref{e:wkw3}) hold.
 Assume that  $\beta$ is sufficiently small.
Since  $\gamma=\beta+1-\alpha$, from (\ref{per3}) it follows
\[
\theta+\theta^{-1}=\frac{2\alpha(\alpha-1)-2(\alpha-1)\beta+\beta^2}{\alpha(\alpha-1)-\alpha\beta}=
2+\frac{2\beta}{\alpha(\alpha-1)}+O(\beta^2).
\] 
Let $\theta=1+\varepsilon$. Then
 $\varepsilon=\sqrt{2/(\alpha^2-\alpha)}\,\sqrt{\beta}+O(\beta) $
 and we can assume that
$\varepsilon$ is  also sufficiently small. Hence $\theta^{-1}=1-\varepsilon+O(\varepsilon^2)$,  $\alpha\theta+\gamma=1+\alpha\varepsilon+O(\varepsilon^2)$,
$m_1=(d_2+\varepsilon)b_2^{-1}$, and 
$m_2=(d_2-\varepsilon)b_2^{-1}+O(\varepsilon^2)$.
Assume that $\varphi^{-1}$ is a $C^2$-function in a neighbourhood of the point $\bar x=d_2b_2^{-1}$. Then
\begin{equation}
\label{e:wkw4}
p_1= A+Bb_2^{-1}\varepsilon+O(\varepsilon^2), \quad
p_2= A-Bb_2^{-1}\varepsilon+O(\varepsilon^2),
\end{equation}
where $A=\varphi^{-1}(\bar x)$ and $B=(\varphi^{-1})'(\bar x)=1/\varphi'(A)$.
Substituting (\ref{e:wkw4}) to (\ref{e:wkw3})  we rewrite (\ref{e:wkw3}) as 
\begin{equation}
\label{e:wkw5}
A+O(\varepsilon)<-2 Bb_2^{-1}<\alpha A+O(\varepsilon).
\end{equation}
Taking sufficiently small $\beta$,  we are also able to check 
condition  (\ref{e:wko}).  Thus, if  $A<-2 Bb_2^{-1}$, $\beta$ is sufficiently small and $\alpha$ is sufficiently large, 
both conditions (\ref{e:wko}) and (\ref{e:wkw3}) are fulfilled
and a non-trivial periodic solution exists.   

\begin{example} 
\label{ex1}
We consider  the two-dimensional model (\ref{d-m}) related to the logistic map, i.e., with 
$f(\mathbf x)=\varphi(x_1+x_2)$ and $\varphi(x)=1-x/K$ for $x\in [0,K]$ and $\varphi(x)=0$ for $x>K$.  
Then $\varphi^{-1}(x)=K(1-x)$, and so $A=K(1-d_2/b_2)$, $B=-K$. 
From (\ref{e:wkw5}) it follows that if 
$2b_2/b_1<b_2-d_2<2$, $d_1=b_1d_2/b_2-\beta$ and $\beta$ is  sufficiently small, then there exists a periodic solution. 
Let us consider a special example with the following coefficients 
$b_1=2.02$, $b_2=0.505$,  $d_1=0.0399$, $d_2=0.01$. Then $b_1/d_1>b_2/d_2$, $\alpha= b_1/b_2=4$, $\beta=\alpha d_2-d_1=0.0001$, $\gamma=-2.9999$,
and $\theta\approx 1.00408$.
Then all conditions hold and a positive periodic solution exists.
If $K=1$ then the periodic sequence is given by  $x_1(2n-1)\approx 0.8482$,
$x_1(2n)\approx 0.8622$,
$x_2(2n-1)\approx 0.1099$,  $x_2(2n)\approx 0.1103$  
 for $n\in\mathbb N$. 
 It is interesting that in this case,  one-dimensional models
(i.e., with the birth and death coefficients $b_1,d_1$ or $b_2,d_2$) have positive and globally stable fixed points
because  $b_i<2+d_i$ for $i=1,2$ (see Section~\ref{ss:loc-stab}).
Hence the two-dimensional  model has a locally stable fixed point $(1-d_1/b_1,0)$ but this point is not globally stable. 
\end{example}

\begin{example} 
\label{ex2}
We consider now the two-dimensional Beverton-Holt model  with harvesting, i.e., a model of type (\ref{d-m}) with 
$f(\mathbf x)=\varphi(x_1+x_2)$ and $\varphi(x)=c/(c+x)$ for $x\in [0,\infty)$, $c>0$.
 A one-dimensional version of this model always has one positive fixed point and this point is globally asymptotically stable 
(see Section~\ref{ss:loc-stab}). 
We have  
$\varphi^{-1}(x)=c/x-c$, and so $A=c(b_2/d_2-1)$, $B=-cb_2^2/d_2^2$. 
Inequality (\ref{e:wkw5})  takes the form
\[
(b_2-d_2)+O(\varepsilon)<2b_2/d_2<\alpha (b_2-d_2)+O(\varepsilon).
\]
The first  inequality is automatically fulfilled for sufficiently small $\beta$.
The second inequality holds if 
$b_1>\dfrac{2b_2^2}{(b_2-d_2)d_2}$
and $\beta$ is sufficiently small
and then a positive periodic solution exists.
\end{example}

\begin{example} 
\label{ex3}
We consider now the two-dimensional model (\ref{d-m}) related to the Ricker  map, i.e., with 
$f(\mathbf x)=\varphi(x_1+x_2)$ and $\varphi(x)=e^{-cx}$ for $x\in [0,\infty)$.
We have  
$\varphi^{-1}(x)=- c^{-1}\ln x$, $(\varphi^{-1})'(x)=- (cx)^{-1}$ and 
so $A=c^{-1}\ln(b_2/d_2)$,  $B=-b_2/(cd_2)$. 
Inequality (\ref{e:wkw5})  takes the form
\[
\ln(b_2/d_2) +O(\varepsilon)<2/d_2<\alpha \ln(b_2/d_2)+O(\varepsilon).
\]
Thus if  $d_2e^{2/(\alpha d_2)}<b_2<d_2e^{2/d_2}$
and $\beta$ is sufficiently small, then a positive periodic solution exists.
Now we give an  example 
when $T$  have a positive periodic point
and  
both one-dimensional models have  globally stable fixed points,
i.e.,  $b_r<d_re^{2/d_r}$  holds (see Section~\ref{ss:loc-stab}).
 Let $b_1=1.0001e^2$, $b_2=b_1/4$,  $d_1=0.9999$, $d_2=0.25$. 
The coefficients  $\alpha= b_1/b_2=4$, $\beta=\alpha d_2-d_1=0.0001$, $\gamma=-2.9999$
are the same as in Example ~\ref{ex1}. Thus $\theta\approx 1.00408$, $\theta^{-1}\approx 0.99594$,  
and $\alpha\theta+\gamma=1.01642$.
For $c=1$ we have $p_i=-\ln m_i$ for $i=1,2$ and we can check that
the periodic sequence is given by  $x_1(2n-1)\approx 1.49009$,
$x_1(2n)\approx 1.51455$,
$x_2(2n-1)\approx 0.000868$,  $x_2(2n)\approx 0.000871$  
 for $n\in\mathbb N$. 
\end{example}
\begin{remark}
\label{r:per}
We have restricted examples only to $f$ of the form $f(\mathbf x)=\varphi(x_1+x_2)$ with the typical $\varphi$ used in the classic 
competition models, to show that these models can have no coexistence equilibrium,  but they can have a positive periodic solution. 
Formula (\ref{ukl-3'}) can be used to find periodic solutions of models with  other $f$'s. 
\end{remark}

\subsection{Stability of  fixed points} 
\label{ss:loc-stab}
In the previous sections  we use some results concerning local and global stability of the transformation $T$
and  to make our exposition self-contained we add a section concerning this subject.
First we look for fixed points of the transformation $T$ and  check their local stability. We assume that $f(\mathbf 0)=1$ and 
\begin{equation}
\label{strict}
L_1>L_2>\dots>L_k.
\end{equation}  
Let $\mathbf x^*$ be a  fixed  point of $T$, i.e., 
$T({\mathbf  x}^*)={\mathbf  x}^*$.
Then 
$x^*_i=0$ or $f(\mathbf x^*)=d_i/b_i=L_i^{-1}$ for $i=1,\dots,k$.
From (\ref{strict}) it follows that  $\mathbf x^*$ is a fixed point of $T$ if 
 $\mathbf x^*=\mathbf x^*_0=\mathbf 0$ or $\mathbf x^*=\mathbf x^*_r=(0,\dots,x^*_r,\dots,0)$, 
 where $r\in\{1,\dots,k\}$ and $f(\mathbf x^*_r)=L_r^{-1}$.
 We assume that the functions $x\mapsto f(0,\dots,x,\dots,0)$ 
are strictly decreasing. 
Then $T$ has exactly $k+1$ fixed points $ \mathbf x^*_r$, $r=0,\dots,k$.
Let $A_r$ be the matrix with
 $a_{ij}^r=\dfrac{\partial T_i}{\partial x_j}({\mathbf  x_r^*})$. We have 
\begin{equation} 
\label{wzpoch}
  \dfrac{\partial T_i}{\partial x_j}({\mathbf  x})
=\delta_{ij}(1-d_i+b_if(\mathbf x))+b_ix_i\dfrac{\partial f}{\partial x_j}(\mathbf x),
\end{equation} 
where  $\delta_{ii}=1$ and $\delta_{ij}=0$ if $i\ne j$. Since 
$f(\mathbf 0)=1$ we obtain $a_{ii}^0=1-d_i+b_i>1$ and $a_{ij}^0=0$ if $i\ne j$,
and therefore  $\mathbf x_0^*=\mathbf 0$ is a repulsive fixed point.
Now we consider a point $\mathbf x_r$ with $r>0$.
Then $f(\mathbf x_r)=d_r/b_r$ and from (\ref{wzpoch}) we obtain
 \[
 a_{ij}^r=\dfrac{\partial T_i}{\partial x_j}(\mathbf  x_r^*)=
 \begin{cases} 
 \delta_{ij}(1-d_i+b_id_r/b_r),  \quad \textrm{if $i\ne r$},\\
 \delta_{ij}+b_rx_r\dfrac{\partial f}{\partial x_r}(\mathbf x_r^*),
  \quad \textrm{if $i= r$}.
\end{cases} 
\]
The matrix $A_r$ has $k$-eigenvalues $\lambda_i$, $i=1,\dots,k$ and $\lambda_i=a^r_{ii}$. 
 We have
\[
\begin{aligned}
\lambda_i&=1-d_i+b_id_r/b_r=1 +b_i(L_r^{-1}-L_i^{-1}),\textrm{ if $i\ne r$},\\
\lambda_r&=1+b_rx_r\dfrac{\partial f}{\partial x_r}(\mathbf x_r^*).
\end{aligned}
\]
 Observe that if $r=1$  then  $\lambda_i\in (0,1)$ for $i>1$. If
 we assume that $-2<b_1x_1\dfrac{\partial f}{\partial x_1}(\mathbf x_1^*)<0$, then also  $\lambda_1\in (0,1)$ and the fixed point $\mathbf x_1^*$ is locally asymptotically stable.
 If $r>1$  then  $\lambda_i>1$ for $i<r$ and, consequently,  
 the fixed point $\mathbf x_1^*$ is not asymptotically stable.
 But if 
  $-2<b_rx_r\dfrac{\partial f}{\partial x_r}(\mathbf x_r^*)<0$,
 the point $\mathbf  x_r^*$ is  locally \textit{semi-stable}, i.e.,  is stable for the transformation $T$ restricted to the set 
\[
X_r=\{{\mathbf  x}\in X \colon \,x_1=\dots=x_{r-1}=0\}.
\]

In the case of the logistic map 
$f(\mathbf x)=1-(x_1+\dots +x_k)/K$
we have $x_r=K(1-d_r/b_r)$, $\dfrac{\partial f}{\partial x_r}(\mathbf x_r^*)=-1/K$,
and conditions for stability (or semi-stability) of $\mathbf x_r^*$  reduce to $b_r<2+d_r$.
If the positive fixed point  $x^*$ of a one-dimensional  logistic map is locally asymptotically stable
  then this point is globally stable, i.e., $T^n(x)\to x^*$, for $x\in (0,K)$. 
Thus, Example~\ref{ex1}  shows that the behaviour of 
a $k$-dimensional logistic map
 and its one-dimensional restrictions
can be different. It can have a locally stable
fixed  point  $\mathbf x_1^*$ but this point can be not globally asymptotically stable. 

Consider a  model with the Beverton-Holt birth rate 
\[
f(\mathbf x)=\dfrac{c}{c+x_1+\dots +x_k}.
\]
Then we have $x_r=c\left(\dfrac{b_r}{d_r}-1\right)$, $\dfrac{\partial f}{\partial x_r}(\mathbf x_r^*)=-\dfrac{c}{(c+x_r)^2}$. 
Conditions for stability (or semi-stability) of $x_r$ reduce to the inequality $b_rx_rc< 2(c+x_r)^2$ 
or, equivalently to $b_r^2/d_r^2>b_r^2/d_r-b_r$, which  always holds because $0\le d_r\le 1$.
The positive fixed point $x^*$ of the one-dimensional map $T$ is globally stable because $x<T(x)<x^*$ for $x\in (0,x^*)$ and $T(x)<x$ for  $x>x^*$.
Example~\ref{ex2}  shows  that a two-dimensional  map have a locally stable
fixed  point  $\mathbf x_1$ but this point can be not globally asymptotically stable.

In the case of the Ricker map $f(\mathbf x)=e^{-c(x_1+\dots +x_k)}$
we have $x_r=\dfrac 1c\ln\dfrac{b_r}{d_r}$, $\dfrac{\partial f}{\partial x_r}(\mathbf x_r^*)=-c\dfrac{d_r}{b_r}$,
and conditions for stability (or semi-stability) of $x_r$  reduce to $cd_rx_r>2$, which takes place when 
$b_r<d_re^{2/d_r}$. The last inequality is also sufficient for global stability of the fixed point (see  e.g. \cite[Th.\ 9.16]{Thieme}). 

\section{Conclusion}
In this paper we consider a discrete time strong competition model.  While in its continuous counterpart 
a population having the maximal turnover coefficient drives the other to extinction, the discrete time model can have no this property. 
We give sufficient conditions for competitive exclusion for a discrete model.
Although this model does not have a coexistence equilibrium,  it can have a positive periodic solution. 
It is interesting that this periodic solution can exist in the case when the restrictions of the model to one dimensional cases 
have globally stable stationary solutions. 
Theorem~\ref{th1} can be also applied to models when the suppression function $f$ depends on other factors,
for example the suppression function $f$ can include  resource density. 
It would be interesting to generalize Theorem~\ref{th1}  to  models with weaker
competition, i.e., when the suppression function is not identical for all  subpopulations, 
or to discrete-continuous hybrid models \cite{GHL,ML} or  to equations on time scales \cite{BP}.

\section*{Acknowledgments}
The author is grateful to Dr. Magdalena Nockowska for several helpful discussions while this work was in progress. 
This research was partially supported by 
the  National Science Centre (Poland)
Grant No. 2014/13/B/ST1/00224.

\end{document}